\documentclass{amsart}
\usepackage{amssymb,amsthm}
\usepackage{hyperref}
\usepackage{pdfsync,bookmark}
\usepackage{todonotes}
\usepackage[T1]{fontenc}
\usepackage[utf8]{inputenc}
\usepackage{lmodern}
\usepackage{thmtools}
\usepackage[capitalize]{cleveref}

\theoremstyle{plain}
\newtheorem{thm}{Theorem}[section]
\newtheorem{prop}[thm]{Proposition}
\newtheorem{lem}[thm]{Lemma}
\newtheorem{cor}[thm]{Corollary}

\theoremstyle{definition}
\newtheorem{defn}[thm]{Definition}
\newtheorem{example}[thm]{Example}
\newtheorem{query}[thm]{Question}

\theoremstyle{remark}

\newtheorem{claim}[thm]{Claim}

\numberwithin{equation}{section}

\renewcommand\epsilon{\varepsilon}
\DeclareMathOperator\Iso{{\mathrm Iso}}
\DeclareMathOperator\Mod{{\mathrm Mod}}
\DeclareMathOperator\Th{{\mathrm Th}}

\crefname{thm}{Theorem}{Theorems}
\def\B{{\mathcal B}}
\def\S{{\mathcal S}}
\def\S{{\mathcal S}}

\def\AS{\mathrm{AS}}
\def\PA{\mathrm{PA}}
\def\TA{\mathrm{TA}}
\def\Pib{\boldsymbol\Pi}
\def\Sigmab{\boldsymbol\Sigma}
\def\Deltab{\boldsymbol\Delta}
\def\Sigmab{\boldsymbol\Sigma}

\newcommand{\A}{\mathcal{A}}
\renewcommand{\phi}{\varphi}
\newcommand{\Pinf}[1]{\Pi^{\mathrm{in}}_{#1}}
\newcommand{\Sinf}[1]{\Sigma^{\mathrm{in}}_{#1}}
\newcommand{\Wwedge}{%
  \mathop{
    \mathchoice{\bigwedge\mkern-15mu\bigwedge}
               {\bigwedge\mkern-12.5mu\bigwedge}
               {\bigwedge\mkern-12.5mu\bigwedge}
               {\bigwedge\mkern-11mu\bigwedge}
    }
}

\newcommand{\Vvee}{%
  \mathop{
    \mathchoice{\bigvee\mkern-15mu\bigvee}
               {\bigvee\mkern-12.5mu\bigvee}
               {\bigvee\mkern-12.5mu\bigvee}
               {\bigvee\mkern-11mu\bigvee}
    }
}
\newcommand{\bx}{\overline x}
\newcommand{\ba}{\overline a}
\newcommand{\concat}{^\smallfrown}
\begin{document}

\title[The complexity of the class of models of theories]
{The Borel complexity of the class\\
of models of first-order theories}

\author[Andrews]{Uri Andrews}
\author[Gonzalez]{David Gonzalez}
\author[Lempp]{Steffen Lempp}
\author[Rossegger]{Dino Rossegger}
\author[Zhu]{Hongyu Zhu}

\address[Andrews, Lempp, Zhu]{Department of Mathematics\\
  University of Wisconsin\\
  Madison, Wisconsin 53706-1325\\
  USA}
\email{\href{mailto:andrews@math.wisc.edu}{andrews@math.wisc.edu}}
\urladdr{\url{http://www.math.wisc.edu/~andrews}}

\email{\href{mailto:lempp@math.wisc.edu}{lempp@math.wisc.edu}}
\urladdr{\url{http://www.math.wisc.edu/~lempp}}

\email{\href{mailto:hongyu@math.wisc.edu}{hongyu@math.wisc.edu}}
\urladdr{\url{https://sites.google.com/view/hongyu-zhu/}}

\address[Gonzalez]{Department of Mathematics\\
 University of California, Berkeley\\
 Berkeley, California 94720-1234\\
 USA}
\email{\href{mailto:david_gonzalez@berkeley.edu}%
{david\_gonzalez@berkeley.edu}}
\urladdr{\url{https://www.davidgonzalezlogic.com/}}

\address[Rossegger]{Institut f\"ur Diskrete Mathematik und Geometrie\\
  Technische Universit\"at Wien\\
  Wiedner Hauptstra{\ss}e 8-10\\
  1040 Wien\\
  AUSTRIA}
\email{\href{mailto:dino.rossegger@tuwien.ac.at}{dino.rossegger@tuwien.ac.at}}
\urladdr{\url{https://drossegger.github.io/}}
\subjclass[2020]{03C62,03C52,03E15}

\keywords{models of arithmetic, Borel complexity, countable models,
descriptive set theory}

\thanks{The first author acknowledges support from the National Science
Foundation under Grant No.\ DMS-2348792. The third author's research was
partially supported by AMS-Simons Foundation Collaboration Grant 626304. The
fourth author's work was supported by the European Union's Horizon 2020
Research and Innovation Programme under the Marie Sk\l{}odowska-Curie grant
agreement No.~101026834 — ACOSE and the Austrian Science Fund
(FWF)~10.55776/P36781. Part of this research was carried out while
the third author was visiting Technical University of Vienna as a guest
professor, and he would like to express his gratitude to this university. --
The authors would like to express their gratitude to Roman Kossak for
connecting them with Ali Enayat and Albert Visser, who were able to establish
complementary results to ours on sequential theories.
}

\begin{abstract}
We investigate the descriptive complexity of the set of models of
first-order theories. Using classical results of Knight and Solovay, we
give a sharp condition for complete theories to have a
$\Pib_\omega^0$-complete set of models.
In particular, any sequential theory (a class of foundational theories
isolated by Pudl\'ak) has a $\Pib_\omega^0$-complete set
of models.
We also give sharp conditions for theories to have a $\Pib^0_n$-complete
set of models.
\end{abstract}

\maketitle

\section{Introduction}\label{sec:intro}

We characterize the possible Borel complexities of the set of models of a
first-order theory. For a single formula~$\phi$, Wadge \cite[I.F.3 and
I.F.4]{Wa83}, using a result by Keisler~\cite{Ke65}, showed that if~$\phi$ is
an $\exists_n$-formula which is not equivalent to a $\forall_n$-formula, then
the set of models of~$\phi$ is a $\Sigmab^0_n$-complete set under Wadge
reduction. We extend this result to considering (possibly incomplete)
first-order theories~$T$ and giving conditions on~$T$ determining the
complexity of $\Mod(T)$, the set of models of~$T$.

We show that a complete theory~$T$ has no $\forall_n$-axiomatization for any
finite~$n$ if and only if $\Mod(T)$ is $\Pib^0_\omega$-complete. Prior to
this result, showing that $\Mod(T)$ is $\Pib^0_\omega$-complete was difficult
even for familiar theories, e.g., Rossegger~\cite{Ro20} asked this for the
theory~TA of true arithmetic. We also show that for any finite~$n$, a
(possibly incomplete) theory~$T$ has a $\forall_n$-axiomatization if and only
if $\Mod(T)$ is~$\Pib^0_n$. If~$T$ does not have a
$\forall_n$-axiomatization, then~$\Mod(T)$ is $\Sigmab^0_n$-hard.

By Vaught's proof~\cite{Va74} of the Lopez-Escobar theorem, showing that the
set of models of~$T$ is~$\Sigmab^0_n$ (or~$\Pib^0_n$) is equivalent to
showing that~$T$ is equivalent to a~$\Sinf{n}$-formula
(or~$\Pinf{n}$-formula, respectively). Also, Wadge's lemma shows that if
the set of models of~$T$ is~$\Sigmab^0_n$ and not~$\Pib^0_n$, then it must be
$\Sigmab^0_n$-complete. Thus, an equivalent way to present our main results
is in terms of when a first-order theory~$T$ is equivalent to a formula
in~$L_{\omega_1\omega}$. For example, it follows that a first-order theory is
equivalent to a $\Pinf{n}$-sentence if and only if it has a
$\forall_n$-axiomatization.

This is related to Keisler's result \cite[Corollary~3.4]{Ke65}
that was recently reproved by Harrison-Trainor and Kretschmer~\cite{HK23}. If a formula in the infinitary logic $L_{\infty\omega}$ is $\Pi^\mathrm{in}_n$ and equivalent to a finitary formula, then it is equivalent to a finitary $\forall_n$-formula.
In fact, our result, when applied to a single formula,
implies the restricted version of this result for $L_{\omega_1\omega}$ formulas via an easy application of compactness. Keisler's result can be interpreted as showing that, though infinitary logic can express much more than finitary logic, it
cannot express things more efficiently, i.e., in fewer quantifier alternations, than
finitary logic.

Interestingly, all three proofs are quite different. Keisler used games and
saturated models, Harrison-Trainor and Kretschmer used arithmetical forcing,
and we use iterated priority constructions. One advantage of our technique is
that, while combinatorially quite complicated, the metamathematics involved
is quite tame. This suggests that our results are a consequence of, if not
equivalent to, compactness. By contrast, Wadge's result that $\Mod(\phi)$ is
$\Sigmab^0_n$-complete if~$\phi$ is $\exists_n$ and not equivalent to a
$\forall_n$-formula relies upon Borel Wadge determinacy. Louveau and
Saint Raymond~\cite{LS87} showed that Borel Wadge determinacy can be proven
in second-order arithmetic. However, all proofs known to date need strong
fragments~\cite{DGHTa}. The second and major advantage is that we consider
theories, not simply formulas.

\section{Preliminaries}

Given a Polish space~$X$, the Borel hierarchy on~$X$ gives us a way to
stratify subsets of~$X$ in terms of their descriptive complexity. A natural
space is the space of countably infinite structures in a countable relational
vocabulary~$\tau$, which we can view as a closed subspace of~$2^\omega$ as
follows. Fix an enumeration of the atomic $\tau$-formulas
$(\phi_i(x_0,\dots,x_i))_{i\in \omega}$; then given a $\tau$-structure~$\A$ with
universe~$\omega$, define its \emph{atomic diagram} by
\[
D(\A)(i)=\begin{cases} 1 & \A \models \phi_i[x_j\mapsto j :j\leq i],\\
                       0 & \text{otherwise.}
         \end{cases}
\]
Let $\Mod(\tau)\subseteq 2^\omega$ be the set of atomic diagrams of
$\tau$-structures with universe~$\omega$. Then it is easy to see that
$\Mod(\tau)$ is a closed subset of Cantor space and thus a Polish space via
the subspace topology.

For a first-order theory~$T$, $\Mod(T)=\{ D(\A): \A\models T\}$ is a
subset of $\Mod(\tau)$, and it is natural to ask how complex
$\Mod(T)$ is in terms of its Borel complexity. It is not hard to see that
$\Mod(T)$ can be at most $\Pib^0_\omega$: It follows from Vaught's
proof~\cite{Va74} of the Lopez-Escobar theorem~\cite{Lo65} that an
isomorphism-invariant subset of $\Mod(\tau)$ is $\Pib^0_\alpha$ if and only
if it is definable by a $\Pinf{\alpha}$-formula in the infinitary
logic~$L_{\omega_1\omega}$ for $\alpha<\omega_1$. Note that we use the
notation $\Pinf{\alpha}$ to refer to formulas in~$L_{\omega_1\omega}$ at that
complexity, and we use~$\forall_n$ to refer to $L_{\omega\omega}$-formulas
with~$n$ alternating quantifier blocks beginning with~$\forall$. Since
$\Mod(T)$ is the set of models of the infinitary formula $\Wwedge_{\phi\in T}
\phi$, and every~$\phi$ is~$\exists_n$ for some $n \in \omega$, we get that
$\Mod(T)$ is at most~$\Pib^0_\omega$.

However, for a fixed theory~$T$, it turns out to be quite difficult to
establish that $\Mod(T)$ is not simpler. The main theorem of this paper
establishes a complete characterization of first-order theories~$T$ such that
$\Mod(T)$ is $\Pib^0_\omega$-complete. To establish this notion of
completeness, we use Wadge reducibility; a subset~$X_1$ of a Polish
space~$Y_1$ is \emph{Wadge reducible} a subset~$X_2$ of a Polish space~$Y_2$
(denoted as $X_1 \leq_W X_2$) if there is a continuous function $f:Y_1\to
Y_2$ such that for all $y\in Y_1$, $y\in X_1$ if and only if $f(y)\in X_2$. A
set $X \subseteq Y$ is \emph{$\Gamma$-hard} for a pointclass~$\Gamma$ if for
every $Z\in \Gamma$, $Z\leq_W X$, and~$X$ is \emph{$\Gamma$-complete} if $X
\in \Gamma$ and~$X$ is $\Gamma$-hard.
Wadge~\cite{Wa83} showed that there is only one \emph{Wadge degree}
(equivalence class induced by $\leq_W$) of a $\Sigmab^0_\alpha$-set which is
not~$\Pib^0_\alpha$ and vice versa. We refer the reader to
Kechris~\cite{Ke95} for more on Wadge reducibility and a thorough
introduction to descriptive set theory.

At last, we mention that we may assume that theories are in the language of
graphs. It is well-known that every structure is interpretable in a graph,
preserving most model-theoretic properties. For example, in~\cite[Section
3.2]{Ro22}, for a given relational vocabulary~$\tau$, a continuous function
$g:\Mod(\tau)\to \Mod(\textrm{Graphs})$ was given such that for all
$\tau$-structures,~$\A$ embeds into~$\B$ if and only if $g(\A)$ elementarily
embeds into $g(\B)$. Easy modifications to the proofs there show that
$\A\equiv\B$ if and only if $g(\A)\equiv g(\B)$, and thus for any
$\tau$-structure~$\A$, $\Mod(\Th(\A))\equiv_W \Mod(\Th(g(\A))$. The same
reduction was given in a different context in~\cite[Proposition 2]{AM15}, the
proof there also shows that $\Mod(\Th(\A))\equiv_W \Mod(\Th(g(\A)))$.

\section{Theories without bounded axiomatization}\label{sec:unbdd}

\begin{defn}
A first-order theory~$T$ is \emph{boundedly axiomatizable} if there is
some~$n$ such that~$T$ has a $\forall_n$-axiomatization.
\end{defn}

Our main result for theories that are not boundedly axiomatizable is the
following

\begin{thm}\label{thm:Piomega}
A complete first-order theory~$T$ has a $\Pib^0_\omega$-complete set of
models if and only if~$T$ is not boundedly axiomatizable.
\end{thm}

In fact, \cref{thm:Piomega} follows directly from the following more
technical fact.

\begin{thm}\label{thm:TTn}
Let~$T$ be any complete first-order theory for which there is a collection of
complete theories $\{T_n\}_{n \in \omega}$ such that for all $n \in \omega$,
$T \neq T_n$ but $T \cap \exists_n = T_n \cap \exists_n$. Then, the collection
of models of~$T$ is $\Pib^0_\omega$-complete. Indeed, for each
$\Pib^0_\omega$-set~$P$, there is a continuous function mapping any $p \in P$
to a model of~$T$, and any $p \notin P$ to a model satisfying~$T_n$ for
some~$n$.
\end{thm}

We show in the next example the necessity of the assumption of completeness
for the theory~$T$ in \cref{thm:Piomega}.

\begin{example}\label{ex:sigmaomega}
Let~$\mathcal{L}_k$ be disjoint relational languages, and for each~$k$,
let~$\phi_k$ be an $\mathcal{L}_k$-sentence which is~$\exists_k$ and not
equivalent to any $\forall_k$-sentence. Let $\mathcal{L} = \bigcup_k
\mathcal{L}_k \cup \{R_i\mid i\in \omega\}$ where each~$R_i$ is a new unary
relation. Let~$T$ say that the set of realizations of each~$R_i$ is disjoint,
and at most one is non-empty. Furthermore, let~$T$ say that any relation
from~$\mathcal{L}_k$ can only hold on tuples from the set of realizations
of~$R_k$. Let~$T$ further say that if~$R_k$ is non-empty, then~$\phi_k$
holds. Clearly, $T$ has no $\forall_n$-axiomatization for
any~$n$, and yet $\Mod(T)$ is~$\Sigmab^0_\omega$.
\end{example}

We defer the proof of Theorem~\ref{thm:TTn} to \cref{sec:proofs}. Here we
state some corollaries of Theorem~\ref{thm:Piomega}:

\begin{cor}\label{cor:PA}
For any completion~$T$ of Peano arithmetic~$\PA$, in particular for true
arithmetic, the set of models of~$T$ is $\Pib^0_\omega$-complete. \qed
\end{cor}

This follows from an observation of Rabin~\cite{Ra61} (which he suspects to
have been known before) that no consistent extension of~PA can be boundedly
axiomatizable.

While \cref{ex:sigmaomega} shows that \cref{thm:Piomega} cannot be
generalized to hold for incomplete theories, for many incomplete theories,
one can use \cref{thm:TTn} to get a similar result. One simply has to find a
suitable completion~$T$ and suitable theories~$T_n$. One example of such a
theory is~PA.

\begin{cor}\label{cor:incPA}
Peano arithmetic has a $\Pib^0_\omega$-complete set of models.
\end{cor}

\begin{proof}
Let $T = \TA$, and let~$T_n$ be a consistent completion of $\TA \cap
\exists_n$, where $\mathsf{B}\Sigma^0_n$, the bounding principle for
$\Sigma^0_n$-formulas, fails. That such~$T_n$ exists for every~$n$ follows
from a result by Parsons~\cite{Pa70}, see also~\cite[Theorem 10.4]{Ka91}.
Using \cref{thm:TTn} with this~$T$ and $(T_n)_{n\in\omega}$, we get that
$\Mod(\PA)$ is $\Pib^0_\omega$-complete.
\end{proof}

Answering our question about other theories of arithmetic, Enayat and Visser
\cite{EVta} showed that no complete sequential theory can be boundedly
axiomatizable. The sequential theories were first defined by
Pudl\'ak~\cite{Pu83} and rephrased by Pakhomov and Visser~\cite{PV22} as
follows:

\begin{defn}\label{def:seq}
Given a theory~$T$, we denote by $\AS(T)$ (\emph{adjunctive set theory}) the
extension of~$T$ by a new binary relation symbol~$\in$ and the axioms
\begin{itemize}
\item
$\AS1$: $\exists x\, \forall y\, (y \notin x)$, and
\item
$\AS2$: $\forall x\, \forall y\, \exists z\, \forall u\, (u \in z
        \leftrightarrow (u \in x \vee u=y))$.
\end{itemize}
A theory is \emph{sequential} if it allows a definitional extension
to~$\AS(T)$.
\end{defn}

Note here that Adjunctive Set Theory does not even require Extensionality.
Pudl\'ak's original definition was in terms of being able to define G\"odel's
$\beta$-function, which then allows for a weak coding of sequences. (Note
that any extension of a sequential theory is again sequential.) Examples of
sequential theories include~$\PA^-$ (by Je\v r\'abek~\cite{Je12}) and
essentially all versions of set theory, but not Robinson's~$\mathrm Q$ (by
Visser~\cite{Vi17}). Thus, Enayat and Visser's result \cite{EVta} yields that
$\Mod(T)$ is $\Pib^0_\omega$-complete for essentially any ``foundational''
complete theory, in particular, any completion of~$\PA^-$.

Finally, note that our reduction in \cref{thm:TTn} produces models of
different theories~$T_n$ for different~$x$ in the $\Sigmab^0_\omega$-outcome,
depending on how we witness that $x \notin P$. We note that it is necessary
to use infinitely many theories in the $\Sigmab^0_\omega$-outcome, since the
union of $\Mod(T_i)$ for finitely many theories~$T_i$ is
always~$\Pib^0_\omega$, and the complement of $P$ could be $\Sigmab^0_\omega$-hard.


\section{Theories with bounded axiomatization}\label{sec:bound}

In this section, we will present results on the Wadge degrees of models of
first-order theories with bounded axiomatization via the quantifier
complexity of their axiomatizations. Our proofs will rely on the following
lemma that will rely on theorems of Knight and Solovay.
We delay its proof to \cref{sec:proofs}.

\begin{lem}\label{lem:corelemma}
Suppose $n \geq 1$ and~$T^+$ and~$T^-$ are distinct complete theories such
that $T^- \cap \exists_n \subseteq T^+ \cap \exists_n$. Then, for any $P \in
\Sigmab^0_n$, there is a Wadge reduction~$f$ such that $f(p) \in \Mod(T^+)$
if $p\in P$, and $f(p) \in \Mod(T^-)$ otherwise. In particular, $\Mod(T^+)$
is $\Sigmab^0_n$-hard, and $\Mod(T^-)$ is $\Pib^0_n$-hard.
\end{lem}

To apply \cref{lem:corelemma} to incomplete theories, we use the
following Lemma which allows us to find completions satisfying the hypotheses
of \cref{lem:corelemma}.

\begin{defn}
A \emph{level-sentence set} for~$\mathcal L$ is either the set
of~$\exists_n$- or the set of $\forall_n$-sentences in~$\mathcal L$ for
some~$n$. 	

For a level-sentence set~$\Lambda$, we let $\neg\Lambda$ be the set of
sentences equivalent to the negation of a sentence in~$\Lambda$.
\end{defn}


\begin{defn}
For a (possibly incomplete) theory~$T$, and a level-sentence set~$\Lambda$,
we let $\Th_\Lambda(T)$ denote $\Lambda\cap\overline{T}$,
where~$\overline{T}$ is the deductive closure of~$T$.
\end{defn}

\begin{lem}\label{lem:hongyu}
Let~$\Lambda$ be a level-sentence set for~$\mathcal{L}$. Let~$A$ be a set of
finitary sentences and~$\phi$ a finitary sentence such that $A \not\vdash
\phi \leftrightarrow \psi$ for any $\psi \in \neg\Lambda$. Then there are
complete consistent theories $T^+ \supseteq A \cup \{\phi\}$ and $T^-
\supseteq A\cup \{\neg \phi\}$ such that $\Lambda\cap T^-\subseteq
\Lambda\cap T^+$. Furthermore, if~$T$ is any theory consistent with $A
\cup \{\phi\} \cup \Th_\Lambda(A\cup \{\neg\phi\})$, then~$T^+$ can be
chosen to contain~$T$.
\end{lem}

\begin{proof}
The lemma follows from the following two claims that allow us to choose
such~$T^+$ and~$T^-$.

\begin{claim}
The theory $A \cup \{\phi\} \cup \Th_\Lambda(A\cup \{\neg\phi\})$ is
consistent.
\end{claim}

\begin{proof}
Suppose that $A \cup \{\phi\} \cup \Th_\Lambda(A\cup \{\neg\phi\})$ is
inconsistent.
By compactness, there is $\psi \in \Th_\Lambda(A\cup \{\neg\phi\})$ such
that $A \cup \{\phi\} \vdash \neg\psi$. But then $A \vdash \phi
\leftrightarrow \neg\psi$ as well as $\neg\psi \in \neg\Lambda$, a
contradiction.
\end{proof}

Now, choose~$T^+$ to be any complete
consistent extension of $A \cup \{\phi\} \cup
\Th_\Lambda(A\cup \{\neg\phi\})$. Observe that if~$T$ is any theory
consistent with $A \cup \{\phi\} \cup \Th_\Lambda(A\cup \{\neg\phi\})$,
then~$T^+$ can be chosen to contain~$T$.

\begin{claim}\label{cl:buildingT-}
The theory $(\neg\Lambda\cap T^+)\cup A\cup \{\neg\phi\}$ is consistent.
\end{claim}

\begin{proof}
Suppose not, then by compactness, there is $\psi \in \neg\Lambda\cap T^+$
such that $A\cup \{\neg\phi\} \vdash \neg\psi$. But then $\neg\psi \in
\Lambda$ and $A \cup \{\neg\phi\} \vdash \neg\psi$, so $\neg\psi \in T^+$,
contradicting that~$T^+$ is consistent.
\end{proof}

Let~$T^-$ be a completion of $(\neg\Lambda\cap T^+) \cup A \cup
\{\neg\phi\}$. Observe that~$T^-$ and~$T^+$ satisfy the lemma.
\end{proof}


\begin{cor}\label{cor:Sigma+PiSimultaneously}
Let~$\Lambda$ be a level-sentence set, and let~$T$ be a theory which is not
$\Lambda$-axiomatizable (i.e., $\Th_{\Lambda}(T)$ does not imply all of~$T$).
Then there are complete theories $T_0,T_1$ such that $T\subseteq T_0$,~$T$ is
inconsistent with~$T_1$, and $\Lambda\cap T_0 \subseteq
\Lambda\cap T_1$.
\end{cor}

\begin{proof}
Let $A = \Th_{\Lambda}(T)$, and let $\phi \in T$ be such that $A \not\vdash
\phi$. Observe that $A \not\vdash \phi \leftrightarrow \psi$ for any $\psi \in
\Lambda$, since otherwise~$\psi$ would be in $\Th_{\Lambda}(T) = A$,
contradicting $A \not\vdash \phi$.


Observe also that $T \cup \Th_{\neg\Lambda}(A \cup \{\neg \phi\})$ is
consistent. Otherwise, there would be a formula $\psi \in \Lambda$ such that
$T \vdash \psi$, thus $\psi \in A$, and $A \cup \{\neg\phi\} \models
\neg\psi$. But then $A \vdash \phi$, which is a contradiction. So, we can
apply \cref{lem:hongyu} to the triple $\neg\Lambda, A, \phi$ to get two
complete theories $T^-\supseteq A \cup \{\neg \phi\}$ and $T^+ \supseteq T$
with $\neg\Lambda\cap T^- \subseteq \neg \Lambda\cap T^+$. Finally, let $T_0
= T^+$ and $T_1 = T^-$.
\end{proof}


\begin{lem}\label{lem:sigma_n}
Let~$T$ be a theory without a $\forall_n$-axiomatization. Then $\Mod(T)$ is
$\Sigmab^0_n$-hard.
\end{lem}

\begin{proof}
By \cref{cor:Sigma+PiSimultaneously}, we have complete theories $T_0
\supseteq T$ and~$T_1$ inconsistent with~$T$ such that $\forall_n \cap T_0
\subseteq \forall_n \cap T_1$. Thus $\exists_n \cap T_1 \subseteq
\exists_n \cap T_0$, and applying \cref{lem:corelemma} shows that $\Mod(T)$
is $\Sigmab^0_n$-hard.
\end{proof}


\begin{lem}\label{lem:pi_n}
Let~$T$ be a theory without an $\exists_n$-axiomatization. Then $\Mod(T)$ is
$\Pib^0_{n}$-hard.
\end{lem}

\begin{proof}
By \cref{cor:Sigma+PiSimultaneously}, we have complete theories $T_0
\supseteq T$ and~$T_1$ inconsistent with~$T$ such that $\exists_n \cap T_0
\subseteq \exists_n \cap T_1$, and applying \cref{lem:corelemma} shows that
$\Mod(T)$ is $\Pib^0_n$-hard.
\end{proof}

\begin{thm}\label{thm:pi_n-ax}
Let $T$ be a theory and $n \in \omega$. Then $\Mod(T) \in \Pib^0_n$ if and
only if~$T$ is $\forall_n$-axiomatizable.
\end{thm}

\begin{proof}
If $\Mod(T) \in \Pib^0_n$, then it is not $\Sigmab^0_n$-hard. So, by
\cref{lem:sigma_n}, it must have a $\forall_n$-axiomatization. On the other
hand, if~$T$ is $\forall_n$-axiomatizable, then $\Mod(T) \in \Pib^0_n$, as
the infinitary conjunction over all sentences in the axiomatization
is~$\Pinf{n}$.
\end{proof}

For $\exists_n$-axiomatizable theories, the situation is not as simple as the
one for $\forall_n$-axiomatizable theories seen in \cref{thm:pi_n-ax}. If~$A$
is a $\exists_n$-axiomatization of~$T$, then $\Mod(T) = \Mod(\psi)$, where
$\psi = \Wwedge_{\phi\in A} \phi$. However, if~$A$ is not a finite
axiomatization, then~$\psi$ is not~$\Sinf{n}$, but
rather~$\Pinf{n+1}$. Combining this with the contrapositive of
\cref{lem:pi_n}, we obtain the following

\begin{prop}\label{thm:sigma_n-ax}
Let~$T$ be a theory and $n \in \omega$. If $\Mod(T)\in \Sigmab^0_n$, then~$T$
is $\exists_n$-axiomatizable. On the other hand, if~$T$ is
$\exists_n$-axiomatizable, then $\Mod(T) \in \Pib^0_{n+1}$. 
\end{prop}

We now give examples of $\exists_n$-axiomatizable theories of different Wadge
degrees showing that the bounds in \cref{thm:sigma_n-ax} cannot be improved.


\begin{example}\label{ex:pi_n+1}
For $k \leq 2$, there are $\exists_k$-axiomatizable $\aleph_0$-categorical
theories with a $\Pib^0_{k+1}$-complete set of models.
\end{example}
\begin{proof}

For $k=1$, let~$\tau$ be the signature consisting of one unary relation
symbol~$P$, and let~$T$ be the theory saying that~$P$ is infinite and
coinfinite.
$T$ is easily seen to be $\exists_1$-axiomatizable, $\aleph_0$-categorical,
and $\Mod(T)$ is $\Pib^0_2$-complete.

For $k=2$, let~$\tau$ be the signature consisting of a single binary relation
symbol~$R$. Let~$T$ say that~$R$ is irreflexive, symmetric and for every~$x$,
there is at most one~$y$ such that $R(x,y)$. Finally, let~$T$ say that there
are infinitely many~$x$ satisfying $\exists y R(x,y)$ and infinitely many~$x$
satisfying $\neg\exists y R(x,y)$. Then~$T$ is $\aleph_0$-categorical and
$\exists_2$-axiomatizable. To show that $\Mod(T)$ is $\Pib^0_3$-complete we
will give a Wadge reduction from the $\Pib^0_3$-complete subset~$P$
of~$2^{\omega\times\omega}$ consisting of all elements having infinitely many
empty columns, i.e., $P=\{ x\in 2^{\omega\times\omega}: \exists^\infty
m\forall n\; x(m,n)=0\}$~\cite[Exercise 23.2]{Ke95}. Given $x\in
2^{\omega\times \omega}$, we produce a structure~$\mathcal M$ with domain the
set $M=\{a_i \mid i\in \omega\}\cup \{b_{\langle j,k\rangle}\mid k \text{ is
least such that $x(j,k)=1$ } \}\cup \{c_i\mid i\in \omega\}\cup \{d_i\mid
i\in \omega\}$. We then set $R(c_i,d_i)$ and $R(a_i,b_{\langle i, k
\rangle})$ for any~$i$ and any~$i,k$ such that $b_{\langle i,k\rangle}$
exists. Note that the domain of~$M$ as defined is a computable set, so there
is an effective bijection with~$\omega$, and thus we may consider this to be a
map from~$2^{\omega\times \omega}$ to $\Mod(\tau)$. It is straightforward to
check that~$\mathcal M$ is a model of~$T$ if and only if $x\in P$.
\end{proof}

\begin{example}\label{ex:sigma_n}
There is a finitely $\exists_3$-axiomatizable $\aleph_0$-categorical theory
with a $\Sigmab^0_3$-complete set of models.
\end{example}

\begin{proof}
Consider the theory of the linear ordering $2\cdot \mathbb{Q} + 1 +
\mathbb{Q}$ together with its successor relation~$S$. This theory is
$\aleph_0$-categorical and is axiomatizable by the axioms for linear
orderings, the definition of the successor relation, the statement that there
is neither a least nor greatest element, and the following
$\exists_3$-formula:
\begin{equation}\tag{$\star$}\label{exalign:s3}
    \begin{gathered}
      \exists x \left[ (\forall y<x) \exists z\,  [y<z<x] \land
        (\forall y>x)( \exists z\, [x<z<y] \land (\forall u>x)
        \neg S(y,u)\right.)\,\wedge
      \\
      \left.(\forall y<x) \left(\exists z
      (S(y,z) \lor S(z,y)) \land
      (\exists u\; S(y,u) \to \forall v\; \neg S(v,y))\right)\right]
    \end{gathered}
\end{equation}
One can easily verify that $L \cong 2 \cdot \mathbb{Q} + 1 + \mathbb{Q}$ for
any countable linear ordering~$L$ satisfying~\eqref{exalign:s3}, thus this is
a finite $\exists_3$-axiomatization for an $\aleph_0$-categorical theory. As
the axiomatization is finite, we get an axiomatization by a single
$\exists_3$-formula and thus its Wadge degree is at most~$\Sigmab^0_3$ by the
Lopez-Escobar theorem. To see hardness, note that it was shown
in~\cite[Theorem 3.3]{GRta} that the isomorphism class of $2 \cdot \mathbb{Q}
+ 1 + \mathbb{Q}$ is $\Sigmab^0_4$-complete in the space of linear orderings
(without successor relation). Now, towards a contradiction, assume it is not
$\Sigmab^0_3$-hard in $\Mod(\{\leq,S\})$. Then \cref{lem:sigma_n} gives a
$\Pinf{3}(\leq,S)$-sentence~$\phi$ such that $\Mod(\phi) = \Iso(2 \cdot
\mathbb{Q} + 1 + \mathbb{Q},\leq,S)$. But clearly~$\phi$ translates into a
$\Pinf{4}(\leq)$-formula, contradicting that $\Iso(2 \cdot \mathbb{Q} + 1 +
\mathbb{Q},\leq)$ is $\Sigmab^0_4$-complete in $\Mod(\{\leq\})$.
\end{proof}

We next show that~$3$ is minimal possible in \cref{ex:sigma_n}. This is
similar to a result of Arnold Miller~\cite{Mi83} that states that no
countable structure can have a $\Sigmab^0_2$-isomorphism class.

\begin{prop}
Let~$\phi$ be a consistent $\Sinf{2}$-sentence. Then~$\phi$ has a finitely
generated model. In particular, if~$T$ is a complete relational theory and
$\Mod(T) \in \Sigmab^0_2$, then $\Mod(T) = \emptyset$.
\end{prop}

\begin{proof}
Suppose~$\phi$ is~$\Sinf{2}$, i.e., of the form $\Vvee_{i\in\omega} \exists
\bx \theta_i(\bx)$, where~$\theta_i$ are conjunctions of
$\forall_1$-sentences. Assume without loss of generality that $(\A,\ba)
\models \theta_i(\ba)$ for some~$i$, then every substructure of $(\A,\ba)$
satisfies $\theta_i(\ba)$ and thus the substructure of~$\A$ generated
by~$\ba$ satisfies~$\phi$. 	

Now assume that~$T$ is a complete relational theory and $\Mod(T) \in
\Sigmab^0_2$. Then by Lopez-Escobar, $\Mod(T) = \Mod(\phi)$ for a
$\Sinf{2}$-formula~$\phi$. It then follows from the above argument that~$T$
has a finite model~$\A$. Hence, $\Mod(T) = \Iso(\A)$ by the completeness
of~$T$.
So,~$T$ does not have a countably infinite model, and $\Mod(T)$ is empty.
\end{proof}

Next we show that Examples~\ref{ex:pi_n+1} and~\ref{ex:sigma_n}  can be
generalized to higher quantifier levels.

\begin{lem}
Let $n \geq 2$. Let~$T$ be an $\exists_n$-axiomatizable theory such that
$\Mod(T)$ is $\Sigmab^0_n$-complete (or $\Pib^0_{n+1}$-complete,
respectively). Let~$T'$ be the $\Delta^0_2$-Marker extension of~$T$ (see
\cite[Lemma 2.8]{AM15}). Then~$T'$ is an $\exists_{n+1}$-axiomatizable theory
such that $\Mod(T)$ is $\Sigmab^0_{n+1}$-complete (or
$\Pib^0_{n+2}$-complete, respectively).
\end{lem}

\begin{proof}
We focus on the case where $\Mod(T)$ is $\Sigmab^0_n$-complete, with the
$\Pib^0_{n+1}$-complete case being similar. Note that $\Mod(T')$
is~$\Sigmab^0_{n+1}$, since from a structure~$B$, it is  $\Delta^0_3(B)$ to
check that it is a $\Delta^0_2$-Marker extension of a structure~$\hat{B}$,
with~$\hat{B}$ being uniformly $\Delta^0_2(B)$. Finally,~$B$ is a model
of~$T'$ if and only if~$\hat{B}$ is a model of~$T$, which is
$\Sigma^0_{n}(B')$. Putting all together, we get that $\Mod(T')$
is~$\Sigmab^0_{n+1}$.	

Since $\Mod(T)$ is $\Sigmab^0_n$-complete, there is a Wadge reduction of $P_n
= \{k\concat p \mid k\in p^{(n)},p\in 2^\omega\}$ to $\Mod(T)$. We will
convert this into a Wadge reduction of $P_{n+1} = \{k\concat p \mid k \in
p^{(n+1)}, p\in 2^\omega\}$ to $\Mod(T')$. Since~$P_{n+1}$ is a
$\Sigmab^0_{n+1}$-complete subset of~$\omega^\omega$, this shows $\Mod(T')$
is $\Sigmab^0_{n+1}$-complete.\footnote{To see that~$P_{n+1}$ is
$\Sigmab^0_{n+1}$-complete, first note that it is~$\Sigmab^0_{n+1}$.
If there was~$D$ such that~$P_{n+1}$ is $\Pi^0_{n+1}(D)$, then~$P_{n+1}$
would be $\Delta^0_{n+1}(D)$. Hence, we would get that for any~$C$
computing~$D$ that $k \in C^{(n+1)}$ if and only if $k\concat \chi_C\in
P_{n+1}$ and this would be $\Delta^0_{n+1}(C)$, hence computable from
$C^{(n)}$. But this would contradict that the Turing jump is proper. So, by
Wadge's lemma,~$P_{n+1}$ is $\Sigmab^0_{n+1}$-complete.}

Let~$g$ be the continuous map reducing~$P_n$ to $\Mod(T)$, and let~$D$ be an
oracle which computes~$g$. Given $k \concat p$, $g(k\concat p')$ is
uniformly computable from $D \oplus p'$. Then $D \oplus p$ can uniformly
compute a copy of the $\Delta^0_2$-Marker extension of $g(k\concat p')$.
This yields the $\Sigmab^0_{n+1}$-hardness of $\Mod(T')$. 	

It is straightforward to check that if~$T$ is $\exists_n$-axiomatizable for
$n \geq 2$, then~$T'$ is $\exists_{n+1}$-axiomatizable.
\end{proof}

Since Marker extensions preserve $\aleph_0$-categoricity and finite
axiomatizability, we can generalize Examples~\ref{ex:pi_n+1}
and~\ref{ex:sigma_n} to higher quantifier levels.

\begin{example}\label{ex:HigherQs}
For every $n \geq 1$, there is an~$\exists_n$-axiomatizable
$\aleph_0$-categorical theory~$T$ such that~$\Mod(T)$ is
$\Pib^0_{n+1}$-complete.

For every $n \geq 3$, there is a finitely~$\exists_n$-axiomatizable
$\aleph_0$-categorical theory~$T$ such that $\Mod(T)$
is~$\Sigmab^0_n$-complete.
\end{example}

\begin{example}\label{ex:diff}
For any $n \geq 3$, there is an $\exists_n$-axiomatizable
$\aleph_0$-categorical theory~$T$ such that $\Mod(T)$ is a properly
$\Deltab^0_{n+1}$-set.
\end{example}

\begin{proof}
Fix~$T_0$ to be an $\exists_{n-1}$-axiomatizable $\aleph_0$-categorical
theory such that $\Mod(T_0)$ is $\Pib^0_{n}$-complete. Fix~$T_1$ to be an
$\exists_n$-axiomatizable $\aleph_0$-categorical theory such that $\Mod(T_1)$
is $\Sigmab^0_{n}$-complete. Let~$T$ have a unary predicate $U$ and say that
the set of elements realizing~$U$ is a model of~$T_0$ and the set of elements
realizing~$\neg U$ is a model of~$T_1$. Then~$T$ is $\exists_n$-axiomatizable
and $\Mod(T)$ is $D_2(\Sigmab^0_{n})$-complete.
\end{proof}



We observe that a special case of \cref{thm:pi_n-ax} implies a case of a
theorem of
Keisler \cite[Corollary~3.4]{Ke65}, recently reproved by Harrison-Trainor and
Kretschmer \cite{HK23}.

\begin{thm}
If a finitary first-order formula~$\phi$ is equivalent to $\psi \in
\Pinf{n}$, then there is a $\forall_n$-formula~$\theta$ such that $\phi
\equiv \theta$.
\end{thm}

\begin{proof}
By adding constants, we may assume that~$\phi$ is a sentence. Since~$\phi$ is
equivalent to~$\psi$, we get that $\Mod(\{\phi\}) \in \Pib^0_n$. Thus,
\cref{thm:pi_n-ax} shows that~$\phi$ has a $\forall_n$-axiomatization.
Compactness implies that~$\phi$ is equivalent to a single
$\forall_n$-sentence.
\end{proof}

Combining our results from this section, we obtain the following
characterization.

\begin{thm}
Let~$T$ be a theory and $n \in \omega$. Then the following are equivalent.
\begin{enumerate}
\item
$T$ has a $\forall_n$-axiomatization but no $\forall_{n-1}$-axiomatization.
\item
The Wadge degree of $\Mod(T)$ is in $[\Sigmab^0_{n-1},\Pib^0_n]$.
\end{enumerate}
\end{thm}

Note that the intervals $[\Sigmab^0_{n-1},\Pib^0_n]$ contain $\aleph_1$-many
different $\Deltab^0_{n}$-Wadge degrees.

\begin{query}
Which~$\Deltab^0_n$-Wadge degrees are the degree of $\Mod(T)$ for some
(complete) finitary first-order theory?
\end{query}

\section{Proofs of the two technical results}\label{sec:proofs}

In the present section, we will prove \cref{thm:TTn} and
\cref{lem:corelemma}. We will first prove \cref{thm:TTn} and then introduce
some minor modifications to the proof to prove \cref{lem:corelemma}.

The proof of \cref{thm:TTn} relies on a theorem of Knight that initially
appeared in~\cite{Kn87} where it is proved using a worker argument, a
technique for (possibly infinite) iterated priority constructions developed
by Harrington. In~\cite{Kn99}, a new proof of this result can be found based
on a version of Ash and Knight's $\alpha$-systems~\cite{AK00}. Let us
introduce all the definitions necessary to state this theorem.

A \emph{Scott set}~$\S$ is a subset of~$2^\omega$ that is closed under Turing
reducibility, join, and satisfies Weak K\"onig's Lemma, i.e., if $T \in \S$
codes an infinite binary tree, then there is a path~$f$ through~$T$ such that
$f \in \S$. An \emph{enumeration} of a countable Scott set~$\S$ is a set $r
\in 2^\omega$ satisfying $\S = \{ r^{[n]} \mid n \in \omega\}$, i.e.,~$\S$
equals the set of columns of~$r$. If $A=r^{[i]}$, then we say that~$i$ is an
$r$-index for~$A$.

We are now ready to state Knight's theorem.

\begin{thm}[{\cite[Theorem~2.5]{Kn99}}]\label{thm:2.5}
Let~$T$ be a complete theory. Suppose $r \le_T x$ is an enumeration of a
Scott set~$\S$, with functions~$t_n$ which are $\Delta^0_n(x)$ uniformly
in~$n$, such that for each~$n$, $\lim_s t_n(s)$ is an $r$-index for $T \cap
\exists_n$, and for all~$s$, $t_n(s)$ is an $r$-index for a subset of $T \cap
\exists_n$. Then~$T$ has a model~$\B$
with $\B \le_T x$.
\end{thm}

We will need the following uniform version of \cref{thm:2.5}. To prove this,
one could simply note that the techniques involved in the proof of
\cref{thm:2.5}---$\alpha$-systems and standard finite-injury
constructions---are uniform. However, as the proofs
in~\cite[Theorem~2.5]{Kn99} are combinatorially quite difficult, we provide
some guidance for the interested reader in the form of a proof sketch.

%

\begin{thm}\label{thm:uniform2.5}
Let~$r$ be an enumeration of a Scott set~$\mathcal S$. Then there exists a
Turing operator $\Gamma:2^\omega\times \omega\to 2^\omega$ which satisfies
the following:

Let~$x$ be any set and~$\Phi_e$ be such that $\lim_s \Phi_e^{x^{(n-1)}}(s)$
is an $r$-index for $T\cap\exists_n$, and for every~$s$,
$\Phi_e^{x^{(n-1)}}(s)$ is an $r$-index for a subset of $T\cap \exists_n$.
Then $\Gamma^{x\oplus r}(e)\in \Mod(T)$.
\end{thm}

\begin{proof}[Proof sketch.]
First, let us point out a notational difference. In \cref{thm:2.5}, the
functions~$t_n$ are required to be~$\Delta^0_n(x)$ uniformly in~$n$, while in
this theorem, the corresponding functions are $\Phi_e^{x^{(n-1)}}$,
where~$\Phi_e$ is a Turing operator that requires the $(n-1)^{\text{th}}$
jump of~$x$ to produce the right output. This discrepancy comes from the
definition of the arithmetic hierarchy. The set~$x^{(n)}$ is the canonical
complete $\Delta^0_{n+1}(x)$-set and is not~$\Delta^0_n(x)$.

We will now point out which parts of the proof of \cref{thm:2.5} guarantee
the required uniformity. In~\cite{Kn99}, Theorem~2.5 (i.e., our
\cref{thm:2.5}) is obtained as a direct consequence of Theorems~2.1
and~2.3. The proof of Theorem~2.1 is a finite injury construction and, as
should be clear to readers familiar with such constructions, is uniform. The
only crucial part is that it needs an effective enumeration of the Scott
set~$\S$, i.e., an enumeration of~$\S$ with additional computable functions
on the indices that witness that~$\S$ is a Scott set. Marker~\cite{MM84}
showed that one can pass from an enumeration of a Scott set to an effective
enumeration, and thus this is not a hindrance.

The proof of Theorem~2.3 uses an $(n+1)$-approximation system relative
to~$\Delta^0_2(X)$. These $(n+1)$-approximation systems are a special version
of $\alpha$-systems developed by Ash and Knight. A standard reference
is~\cite{AK00}. The proof in Knight~\cite{Kn99} just shows that the premises
of Theorem~2.3 are sufficient to obtain an $(n+1)$-approximation system
satisfying the premises of \cite[Corollary 4.4]{Kn99}. This corollary
essentially says that there is a function~$E$ that takes paths~$\pi$ through
a special tree and produces a $\Sigma^0_2(x)$-object~$E(\pi)$. This~$E$ is
easily seen to be continuous from its definition, see \cite[p268, third
paragraph]{Kn99}. But one issue is that we cannot obtain~$\pi$ from~$x$,
rather~$\pi$ is the limit of approximations~$\pi^n$ which
are~$\Delta^0_n(x)$, uniformly in~$n$. Thankfully, as is explained in the
proof of \cite[Theorem~4.1]{Kn99} that gives rise to \cite[Corollary 4.4]{Kn99}, $E(\pi)
= E(\pi^1)$ and~$\pi^1$ can be uniformly computed from~$x$.
As $E(\pi)=E(\pi^1)$ is what we are after, we get the desired continuity.
\end{proof}

\subsection{Proof of Theorem \ref{thm:TTn}}

We fix a $\Pib^0_\omega$-set~$P\subseteq 2^\omega$, theories~$T$ and
$\{T_n\}_{n \in \omega}$, and show how to obtain the Scott set~$\S$ and
functionals~$t_n$ such that, using an input $p \in 2^\omega$, we satisfy
\cref{thm:uniform2.5} and thus output a model of~$T$ if $p \in P$ and a model
of some~$T_n$ otherwise.


Given a $\Pib^0_\omega$-set~$P$, we can fix a decreasing sequence of
$\Pib^0_n$-sets~$P_n$ such that $P = \bigcap_{n \ge 1} P_n$. Let~$c$ be
strong enough so that each~$P_n$ is uniformly~$\Pi^0_n(c)$.


Next, we fix an enumeration~$r$ of a Scott set~$\S$ containing $T \cap
\exists_k$ for each~$k$ and $\{T_n \cap \exists_k\}_{n,k \in\omega}$. Let $y
= c \oplus r \oplus T \oplus \bigoplus_{n \in \omega} T_n$. We will describe
a computation~$\Phi_e$
satisfying the hypotheses of \cref{thm:uniform2.5}, namely that $\Phi^{(y
\oplus p)^{(n-1)}}_{e}(s)$ will be constant in $s$ and is an $r$-index for
$T_p \cap \exists_n$ for a complete theory~$T_p$.
Furthermore, we will ensure that if $p \in P$ then $T_p = T$; and if $p
\notin P$ then~$T_p=T_n$ for some $n \in \omega$. Note that we are applying
\cref{thm:uniform2.5} with the oracle $x=y\oplus p$.

We describe the index $e$ by giving a uniform method of computing an
$r$-index for $T_p \cap \exists_n$ from $(y \oplus p)^{(n-1)}$. For $n=1$, we
output a fixed index for $T \cap \exists_1$. For $n \ge 2$, $\Phi^{(y \oplus
p)^{(n-1)}}_{e}$ depends on whether $p \in P_{n-1}$.
Since $P_{n-1}$ is uniformly~$\Pi^0_{n-1}(c)$, it is uniformly computable in
$(c\oplus p)^{(n-1)}\leq_T (y\oplus p)^{(n-1)}$ to determine membership
of~$p$ in~$P_{n-1}$.

For $n \geq 2$, let~$\Phi_e$ be the algorithm defined as follows. Given an
oracle of the form $(y \oplus p)^{(n-1)}$, let~$k_0$ be the least $k<n$ such
that $p \notin P_k$ if such~$k$ exists, and~$n$ otherwise. Let $\Phi_{e}^{(y
\oplus p)^{(n-1)}}(s)$ output the least $r$-index of $T_{k_0} \cap
\exists_n$. Note that finding this index is not effective in~$r$
and~$T_{k_0}$ but is effective in $(r\oplus T_{k_0})'\leq_T y'\leq_T (y\oplus
p)^{(n-1)}$ (needing one jump here is why we treat the case $n=1$
differently).


We have just produced a uniform sequence of computations, thus we can find a
single index~$e$ (note that~$y$ and~$p$ are used in the oracle, but not in
identifying the index) such that if $p\in P$ then for every~$n$,
$\Phi_e^{(y\oplus p)^{(n-1)}}$ satisfies the conditions in
\cref{thm:uniform2.5} for $T$ and if $p\not\in P$, then the conditions in
\cref{thm:uniform2.5} are satisfied for some~$T_n$. Thus, the Turing
operator~$\Gamma$ in \cref{thm:uniform2.5} gives a Wadge reduction from~$P$
to $\Mod(T)$, as $\Gamma^{y\oplus p}\in \Mod(T)$ if and only if $p\in P$.

\subsection{Proof of Lemma \ref{lem:corelemma}}\label{sec:corelemmaproof}

The proof is almost the same as the proof of \cref{thm:TTn}, except that one
of our approximation functions will not be constant. The difference is in how
we obtain the index~$e$ to apply \cref{thm:uniform2.5}. We let $y=c \oplus
r\oplus f \oplus g$, where~$P$
is~$\Sigma^0_n(c)$,
and~$r$ is an enumeration of a Scott set which contains both~$T^+$ and~$T^-$;
now~$f$ and~$g$ are functions such that $f(n)$ is an $r$-index for $T^-\cap
\exists_n$ and $g(n)$ is an $r$-index for $T^+\cap \exists_n$. Let $P =
\bigcup_{i\in \omega} P_i$, where the~$P_i$ are
uniformly~$\Delta^0_n(c)$.

Here, for $k<n$, we let~$\Phi_e^{(y\oplus p)^{(k-1)}}(s)$ output $f(k)$,
which is an $r$-index for $T^-\cap \exists_k = T^+\cap \exists_k$.
We now describe the algorithm to compute~$\Phi_e^{(y\oplus p)^{(n-1)}}(s)$
for $k = n$.
First, check whether $p \in \bigcup_{t<s} P_t$. Note that this is uniformly
computable from $(c\oplus p)^{(n-1)}$. If $p \notin \bigcup_{t<s} P_t$, we
output
$f(n)$, which is an $r$-index for
$T^- \cap \exists_n$. If $p \in \bigcup_{t<s} P_t$, we output
$g(n)$, which is an $r$-index for
$T^+ \cap \exists_n$. Note that the value of $\Phi_e^{(y\oplus
p)^{(n-1)}}(s)$ may change at most once as~$s$ increases. For $k > n$, we let
$\Phi_e^{(y\oplus p)^{(k-1)}}$ be the algorithm that checks whether $p \in P$
and outputs $g(k)$, an $r$-index for $T^+ \cap \exists_k$ if $p\in P$, and
outputs $f(k)$, the index for $T^- \cap \exists_k$ otherwise. The fact that
$T^- \cap \exists_n \subseteq T^+ \cap \exists_n$ guarantees that if $p\in
P$, then $\Phi_e$ satisfies the conditions for \cref{thm:uniform2.5} with
$T=T^+$, and that if $p\not\in P$, $\Phi_e$ satisfies the conditions for
\cref{thm:uniform2.5} with $T=T^-$. Thus,~$\Gamma$ gives a Wadge reduction
from~$P$ to $\Mod(T^+)$.

\end{document}